\documentclass[a4paper,12pt]{article}
\usepackage{amssymb,amsmath,amsthm,color}

\numberwithin{equation}{section}

\numberwithin{equation}{section}

\newtheorem{thm}{\sc Theorem}[section]

\newtheorem{lem}[thm]{\sc Lemma}

\newtheorem{rem}[thm]{\sc Remark}

\newcommand{\h}{{\mathbb H}}

\newcommand{\R}{{\mathbb R}}

\newcommand{\ve}{\varepsilon}
\newcommand{\holder}{H\"older}

\title{Sobolev type inequalities 
for fractional maximal functions and Riesz potentials in half spaces}

\author{Yoshihiro Mizuta and Tetsu Shimomura}
\date{}

\begin{document}
\maketitle

\begin{abstract}
In this paper, we study Sobolev type inequalities for fractional maximal functions $M_{\h,\nu}f$ and Riesz potentials $I_{\h,\alpha} f$ of functions in  weighted Morrey spaces of the double phase 
functional $\Phi(x,t) = t^{p} +  (b(x) t)^{q}$ in the half space, where $1<p<q$ and $b(\cdot)$ is non-negative, bounded and H\"older continuous of order $\theta \in (0,1]$. We also show that the Riesz potential operator $I_{\h,\alpha}$ embeds from weighted Morrey space of the double phase functional $\Phi(x,t)$ to weighted Campanato spaces. Finally, we treat the similar embedding for Sobolev functions. 
 \end{abstract}


\section{Introduction}

Morrey spaces were introduced by C. B. Morrey in 1938 to study the existence and regularity of partial differential equations (\cite{Mo}). We also refer to \cite{P}. The boundedness of the maximal operator was studied on Morrey spaces in \cite{CF}. For Herz-Morrey-Orlicz spaces  on the half space, see \cite{MOS3}. The boundedness of the fractional maximal operator was studied on Morrey spaces in \cite{FR}. For local Morrey-type spaces, see \cite{BGGM}. There are many related results; see e.g. \cite{AHS, BH, CCF, KiS, KrK, L, MMOS, MNOS, SST}. There has been an increasing interest in  Sobolev spaces; see \cite{CF1, DHHR, HH} and so on. 

For a locally integrable function $f$ 
 on the half space $\h = \{ x=(x',x_n) \in \R^{n-1} \times \R^1: 
x_n > 0\}$ and $\nu$ with $0\le \nu \le n$, 
the  fractional Hardy-Littlewood 
 maximal function $M_{\h,\nu}f$ is defined by
\[
	M_{\h,\nu}f(x)=\sup_{\{r >0: B(x,r)\subset \h\}} 
	\frac{r^\nu}{|B(x,r)|}
	\int_{B(x,r)} |f(y)|\,  dy,
\]
where $B(x,r)$ is the ball in $\R^n$ centered at  $x$ of radius $r>0$ 
and  $|B(x,r)|$ denotes its Lebesgue measure.  The mapping $f \mapsto M_{\h,\nu} f$ is called the fractional central maximal operator. When $\nu=0$, we write $M_\h f$ instead of $M_{\h,0}f$. 

 In view of the well-known theorem by Hardy and Littlewood,  
  the usual  maximal operator  is bounded in $L^p(\h)$. 
  However, this  is not always true 
  in the weighted $L^p$ spaces, 
  as will be seen in Remarks \ref{rem:2.0} and \ref{rem:2.1} below. 
  To conquer difficulties, we consider the local maximal operators; 
 for an application of the local maximal operators, see \cite{KL} . 
    
   In the previous paper \cite{MS5}, we established a Sobolev type inequality
 for the fractional maximal function  $M_{\h,\nu} f$  
 in weighted Morrey spaces. In fact, the following result is shown in \cite[Theorem 2.1]{MS5}:
\vspace{3mm}

{\sc Theorem A}.
\sl
Let $p > 1$, $1/p^* = 1/p - \nu/\sigma > 0 $ 
and  $0<\sigma < (n+1)/2$. 
Suppose $ \beta < (n+1)/(2p')$, where $1/p + 1/p' = 1$. 
Then there exists a constant $C>0$ such that
\begin{eqnarray*}
\sup_{\{r > 0: B(x,r) \subset \h\}}  
\frac{r^\sigma}{|B(x,r)|}\int_{B(x,r)}   \left( z_n^\beta M_{\h,\nu} f(z)  \right)^{p^*} 
  dz  &\le& C 
\end{eqnarray*}
when $\displaystyle \sup_{r >0,x\in \h} 
	\frac{r^\sigma}{|B(x,r)|}\int_{\h\cap B(x,r)} 
	\left( |f(y)| y_n^\beta \right)^p dy \le 1$. 
\rm
\vspace{3mm}

This is not true for the usual fractional maximal function $M_{\nu} f$ defined by
\[
	M_{\nu}f(x)=\sup_{r >0} 
	\frac{r^\nu}{|B(x,r)|}
	\int_{\h\cap B(x,r)} |f(y)|\,  dy; 
\]
see Remarks \ref{rem:2.0} and \ref{rem:2.1}.

The double phase functional was introduced by Zhikov \cite{Z} in the 1980s. Regarding regularity theory of differential equations, Mingione and collaborators \cite{BCM1, BCG, CM, CM2} investigated a double phase functional 
$$
\hat{\Phi}(x,t) = t^p + a(x) t^q, \ x \in \R^N, \ t \ge 0, 
$$ 
where $1 < p < q$, $a(\cdot)$ is nonnegative, bounded and H\"older continuous of order $\theta \in (0,1]$. See \cite{HO, Shin} for Calder\'on-Zygmund estimates, \cite{MMOS3, MOS3} for Sobolev inequalities, \cite{MS4} for Hardy-Sobolev inequalities and \cite{MNOS3, MNOS4} for Campanato-Morrey spaces for the double phase functional. We refer to for instance \cite{BL, CS, FM, FO1, HHK, MS3} and references therein for other recent works. 

In the present paper, relaxing  the continuity of $a(\cdot)$, we consider 
the double phase functional 
\[ 
\Phi(x,t) = t^{p} +  (b(x) t)^{q},
 \]
where $1<p<q$ and
$b(\cdot)$ is non-negative, bounded and H\"older continuous of order $\theta \in (0,1]$ (cf. \cite{CM}); if we write $\Phi(x,t) = t^{p} +  a(x) t^{q}$ with $a(x) = b(x)^q$, then $a$ is not always  H\"older continuous of order $\theta q$ when $\theta q > 1$.

In connection with Theorem A above, 
our first aim in this paper is to give  Sobolev type inequalities 
 for $M_{\h,\nu} f$ of functions in weighted Morrey spaces of the double phase functional $\Phi(x,t)$ (Theorem  \ref{thm:2.1}). 
We are mostly interested in functions $f$ 
 for which $M_{\nu} f = \infty$; see Remark \ref{rem:2.1} given below.

For $0 < \alpha < n$ and a locally integrable function $f$ on $\h$, we define the Riesz potential of order $\alpha$ in $\h$ by 
\[
I_{\h,\alpha} f(x) = \int_{B(x,x_n)} |x-y|^{\alpha-n} f(y) dy. 
\]
Our arguments are applicable to the study of Sobolev's inequalities for $I_{\h,\alpha} f$ (Theorem \ref{thm:3.1}), which have not been found in the literature. The sharpness of Theorem  \ref{thm:3.1} will be discussed in Remark \ref{rem:3.1}.
 
In Section \ref{s4}, we are concerned with Sobolev's inequalities for $I_{\h,\alpha} f$ of functions in  weighted Morrey spaces of the double phase functional $\Phi(x,t)$ (Theorem  \ref{thm:3.2}). 

In Section \ref{s5}, we treat the case $\sigma=\alpha p$. In fact, we show that $I_{\h,\alpha}$ embeds from weighted Morrey spaces to 
 weighted Campanato spaces in the case $\sigma=\alpha p<(n+1)/2$ (Theorem \ref{thm:3.3}). Further, we show that $I_{\h,\alpha}$ embeds from weighted Morrey spaces of the double phase functional $\Phi(x,t)$ to   weighted Campanato spaces in the case $\sigma=\alpha q= (\alpha+\theta) p$ (Theorem \ref{thm:3.4}). 

In the final section, we show  the embedding for Sobolev functions   
in the same frame (Theorems \ref{thm:A1} and \ref{thm:A2}). 

Throughout this paper, let $C$ denote various constants 
independent of the variables in question.   
The symbol $g \sim h$ means that $C^{-1}h\le g\le Ch$ for some constant $C>0$.

\section{Boundedness of fractional maximal operators for double phase functionals}

Throughout this paper, let 
\[
p > 1 \quad \text{and} \quad \sigma > 0 .
\]

Our aim in this section is to study the boundedness of the fractional central maximal operator  $M_{\h,\nu} $ in weighted Morrey spaces of the double phase functional $\Phi(x,t)$.

\begin{thm} \label{thm:2.1}
\sl
Let  $p > 1$, $1/q = 1/p - \theta/\sigma$, 
$1/p^* = 1/p - \nu/\sigma > 0$ and  
$1/q^* = 1/q - \nu/\sigma > 0$. Set 
\[
{\Phi}^*(x,t) = \Phi_{p^*,q^*}(x,t) 
=  t^{p^*} + (b(x) t)^{q^*} .
\]
Suppose $ \beta < (n+1)/(2p')$ 
and $0<\sigma < (n+1)/2$, where $1/p + 1/p' = 1$. 
Then there exists a constant $C>0$ such that
\begin{eqnarray*}
\sup_{\{r > 0: B(x,r) \subset \h\}}  
\frac{r^\sigma}{|B(x,r)|}\int_{B(x,r)}   
\Phi^*(z, z_n^\beta M_{\h,\nu} f(z)  )   dz  \le C 
\end{eqnarray*}
when $\displaystyle \sup_{r >0,x\in \h} 
	\frac{r^\sigma}{|B(x,r)|}\int_{\h\cap B(x,r)} 
	\Phi(y, |f(y)| y_n^\beta) dy \le 1$. 
\end{thm}

\begin{rem} \label{rem:2.0}
\rm
In Theorem \ref{thm:2.1}, the assumption that $B(x,r) \subset \h$ is needed. See \cite[Remark 2.9]{MS5}.
\end{rem}

Before a proof of Theorem \ref{thm:2.1}, we recall some lemmas from \cite{MS5, FR}. 

\begin{lem}[{\cite[Lemma 2.3]{MS5}}]  \label{lem:1.2}
\sl
 For $\ve > (n-1)/2$, set
 \begin{eqnarray*}
I(x) &=& \int_{B(x,x_n)}   y_n^{\ve - n}  dy .
\end{eqnarray*}
Then there exists a constant $C>0$ such that
\begin{eqnarray*}
I(x) &\le& C  x_n^{\ve}   .
\end{eqnarray*}
\end{lem}

\begin{rem} \label{rem:2.1}
\rm
 If $f(y) = |y_n|^{-1}$, then 
 \begin{eqnarray*}
M_\nu f(x) = \infty 
\end{eqnarray*}
for all $x\in \R^n$. However, 
 \begin{eqnarray*}
M_{\h,1} f(x) \le C 
\end{eqnarray*}
for $x\in \h$, which is shown by Lemma \ref{lem:1.2}. 
\end{rem}


\begin{lem}[{\cite[Lemma 2.4]{MS5}}]  \label{lem:1.3}
\sl
 For $\ve < (n-1)/2$, set
 \begin{eqnarray*}
J(y) &=& \int_{\{x\in \h : |x-y| <x_n\}} 
   x_n^{\ve - n}  dx .
\end{eqnarray*}
Then there exists a constant $C>0$ such that
\begin{eqnarray*}
J(y) &\le& C  y_n^{\ve}     .
\end{eqnarray*}
\end{lem}

We know the following result on the boundedness of $M_\nu$.

\begin{lem} [{\cite[Lemma 4]{FR}, cf. \cite[Corollary 2]{AHS}}] \label{lem:2.1}
\sl
 Let $1/p^* = 1/p - \nu/\sigma > 0$. 
 Then there exists a constant $C>0$ such that
\begin{eqnarray*}
\sup_{r > 0}  
\frac{r^\sigma}{|B(x,r)|}\int_{B(x,r)}   
\left(  M_\nu g(x)  \right)^{p^*} 
  dx  &\le& C 
\end{eqnarray*}
when $\displaystyle \sup_{r >0,x\in \h} 
	\frac{r^\sigma}{|B(x,r)|}\int_{B(x,r)} |g(y)|^p dy \le 1$. 
\end{lem}

\begin{lem}[{\cite[Lemma 2.7]{MS5}}]  \label{lem:2.4}
\sl
  Set
 \begin{eqnarray*}
K(x) &=& \frac{x_n^\nu}{|B(x,x_n)|} \int_{B(x,x_n)} |f(y)| dy .
\end{eqnarray*}
Suppose 
$1/p^* = 1/p - \nu/\sigma > 0$, 
$0<\sigma < (n+1)/2$ and $ \beta < (n+1)/(2p')$. 
Then there exists a constant $C>0$ such that
\begin{eqnarray*}
\sup_{0<r<x_n}  \frac{r^\sigma}{|B(x,r)|}    \int_{B(x,r)}
\left( K(z) z_n^\beta \right)^{p^*} dz \le C  
\end{eqnarray*}
when $\displaystyle 
\sup_{x\in \h}  \frac{x_n^\sigma}{|B(x,x_n)|}    \int_{B(x,x_n)} \left(
|f(y)| y_n^{\beta} \right)^p  dy \le 1$.
\end{lem}

Let us prove Theorem \ref{thm:2.1}.

\begin{proof}[{\it Proof of Theorem \ref{thm:2.1}}]
Let $f$ be a measurable function on $\h$ such that
$$
\sup_{r >0,x\in \h} 
	\frac{r^\sigma}{|B(x,r)|}\int_{\h\cap B(x,r)} 
	\Phi(y, |f(y)| y_n^\beta) dy \le 1.
$$ 
First we see from \cite[Theorem 2.1]{MS5} that
$$
\sup_{\{r > 0: B(x,r) \subset \h\}}  \frac{r^\sigma}{|B(x,r)|}\int_{B(x,r)}  
 (z_n^\beta M_{\h,\nu} f(z))^{p^*}  dz \le C.
$$
Next we show that
$$
\sup_{\{r > 0: B(x,r) \subset \h\}}  \frac{r^\sigma}{|B(x,r)|}\int_{B(x,r)}  
 (b(z) z_n^\beta M_{\h,\nu} f(z))^{q^*}  dz \le C.
$$
Note that 
  \begin{eqnarray*}
&& \frac{r^\nu}{|B(x,r)|} b(x) \int_{B(x,r)} |f(y)| dy  \\
&=& 
\frac{r^\nu}{|B(x,r)|}\int_{B(x,r)} \{b(x)-b(y)\} |f(y)| dy  
+
\frac{r^\nu}{|B(x,r)|}  \int_{B(x,r)} b(y) |f(y)| dy  \\
&\le&C\frac{r^{\nu+\theta}}{|B(x,r)|} \int_{B(x,r)} |f(y)|  dy 
+
\frac{r^\nu}{|B(x,r)|}  \int_{B(x,r)} b(y) |f(y)| dy,
\end{eqnarray*}
so that
\begin{eqnarray*}
L_1(x) &=& \sup_{0<r<x_n/2}
\frac{r^{\nu}}{|B(x,r)|} b(x) \int_{B(x,r)} |f(y)|  dy 
\\
&\le&  \sup_{0<r<x_n/2} C x_n^{-\beta}  \frac{r^{\nu+\theta}}{|B(x,r)|} \int_{B(x,r)} |f(y)|y_n^{\beta}  dy \\
&& + \sup_{0<r<x_n/2} C x_n^{-\beta}
\frac{r^\nu}{|B(x,r)|}  \int_{B(x,r)} b(y) |f(y)|y_n^{\beta}  dy \\
&\le& C x_n^{-\beta}M_{\nu+\theta} g(x) +C x_n^{-\beta} M_\nu h(x),
\end{eqnarray*}
where $g(y) = |f(y)|y_n^\beta  \chi_{\h}(y)$ and $h(y) =b(y) |f(y)|y_n^\beta  \chi_{\h}(y)$. We have \begin{eqnarray*}
L_2(x) &=& \sup_{x_n/2<r<x_n}
\frac{r^{\nu}}{|B(x,r)|} b(x) \int_{B(x,r)} |f(y)|  dy 
\\
&\le&C\frac{x_n^{\nu+\theta}}{|B(x,x_n)|} \int_{B(x,x_n)} |f(y)|  dy 
+
C \frac{x_n^\nu}{|B(x,x_n)|}  \int_{B(x,x_n)} b(y) |f(y)| dy  \\
&=& C\{L_{21}(x) + L_{22}(x)\}.
\end{eqnarray*}
Hence 
\begin{equation}\label{(1)}
b(z)  z_n^{\beta}M_{\h,\nu} f(z) \le C M_{\nu+\theta} g(z) +C M_\nu h(z)+  C L_{21}(z) z_n^{\beta}+  C L_{22}(z) z_n^{\beta}
\end{equation}
for $z \in B(x,r)$. By Lemma \ref{lem:2.4}, we obtain 
$$
\sup_{0<r<x_n} \frac{r^\sigma}{|B(x,r)|}  \int_{B(x,r)} \left( L_{21}(z) z_n^\beta \right)^{q^*} dz \le C
$$
and 
$$
\sup_{0<r<x_n} \frac{r^\sigma}{|B(x,r)|}  \int_{B(x,r)} \left( L_{22}(z) z_n^\beta \right)^{q^*} dz  \le C.
$$

By (\ref{(1)}) and Lemma \ref{lem:2.1}, we obtain for $r > 0$ such that $B(x,r) \subset \h$ 
\begin{eqnarray*}
&&    \frac{r^\sigma}{|B(x,r)|}\int_{B(x,r)}  
 (b(z) z_n^\beta M_{\h,\nu} f(z))^{q^*}  dz \\
&\le& C \frac{r^\sigma}{|B(x,r)|}\int_{B(x,r)}(M_{\nu+\theta} g(z))^{q^*}  dz
+C \frac{r^\sigma}{|B(x,r)|}\int_{B(x,r)}(M_{\nu} h(z))^{q^*}  dz \\
&& +C \frac{r^\sigma}{|B(x,r)|}    \int_{B(x,r)} \left( L_{21}(z) z_n^\beta \right)^{q^*} dz +C \frac{r^\sigma}{|B(x,r)|}    \int_{B(x,r)} \left( L_{22}(z) z_n^\beta \right)^{q^*} dz \\
&\le& C.
\end{eqnarray*}
Thus the proof is completed.
\end{proof}

\section{Sobolev's inequality}

In this section we are concerned 
with  the Riesz potential of order $\alpha$ in $\h$ defined by 
\[
I_{\h,\alpha} f(x) = \int_{B(x,x_n)} 
|x-y|^{\alpha-n} f(y) dy , 
\]
where $0 < \alpha < n$.

For $0 < r < x_n/2$, we see that 
\begin{equation}\label{eq:3.2}
  J_1(x)    =  \int_{B(x,x_n/2)} |x-y|^{\alpha-n} |f(y)| dy
 \le C x_n^{-\beta}  \int_{\h} |x-y|^{\alpha-n} g(y) dy,
\end{equation}
 where $g(y) = |f(y)| y_n^\beta \chi_{\h}(y)$, as before.

For $0<\alpha<n$ and a locally integrable function $g$ on $\R^n$, we define the usual Riesz potential $I_\alpha g$ of order $\alpha$ by 
\[
I_\alpha g(x) = \int_{\h} |x-y|^{\alpha-n} g(y) \, dy .
\]

The following result is due to Adams \cite{A}.

\begin{lem}[{\it Sobolev's inequality for Morrey spaces}]
 \label{lem:1.1}
    \sl
Let $1/p^* = 1/p - \alpha/\sigma > 0$. Suppose $\alpha p<\sigma \le n$. 
Then there exists a constant $C>0$ such that 
\[
    \frac{r^{\sigma}}{|B(x,r)|}\int_{B(x,r)}
    \left| I_\alpha g(z) \right|^{p^*}\, dz  \le C  
\]
for all $x\in \R^n$, $r>0$ and measurable functions  $g$ on $\R^n$ 
with 
\[
 \sup_{r >0,x\in \R^n} 
	\frac{r^\sigma}{|B(x,r)|}\int_{B(x,r)} |g(y)|^p dy \le 1. 
	\]  
\end{lem}

\begin{lem}\label{lem:3.1}
\sl
Suppose 
\[
 \sup_{r >0,x\in \h} 
	\frac{r^\sigma}{|B(x,r)|}\int_{\h\cap B(x,r)} 
	\left( |f(y)| y_n^\beta \right)^p dy \le 1 . 
	\]
		If $1/p^* = 1/p - \alpha/\sigma > 0$, then there exists a constant $C>0$ such that
	  \begin{eqnarray*}
	 \sup_{\{r >0: B(x,r)\subset \h\}}  
	\frac{r^\sigma}{|B(x,r)|}\int_{B(x,r)} 
	\left( z_n^{\beta} J_1(z)  \right)^{p^*} dz \le C   . 
 \end{eqnarray*} 
\end{lem}

\begin{proof}
Set  $g(y) = |f(y) | y_n^\beta \chi_{\h}(y)$ for simplicity. By (\ref{eq:3.2}), we have
$$
  z_n^{\beta} J_1(z)   \le C I_\alpha g(z),
$$
so that by Lemma \ref{lem:1.1}
$$
 \frac{r^\sigma}{|B(x,r)|}\int_{B(x,r)} 
	\left( z_n^{\beta} J_1(z)  \right)^{p^*} dz  \le C  \frac{r^{\sigma}}{|B(x,r)|}\int_{B(x,r)}
    \left| I_\alpha g(z) \right|^{p^*}\, dz \le C,
$$
as required. 
\end{proof}

Next let us treat 
	  \begin{eqnarray}
J_2(x) &=& \int_{B(x,x_n)\setminus B(x,x_n/2)} |x-y|^{\alpha-n} |f(y)| dy \nonumber \\
&\le& C x_n^{\alpha-n}  
\int_{B(x,x_n)\setminus B(x,x_n/2)}  |f(y)| dy . \label{eq:3.3}
 \end{eqnarray}  
Here we prepare the following lemma. 

\begin{lem}\label{lem:3.4}
\sl
Suppose $1/p^* = 1/p - \alpha/\sigma > 0$, $\beta < (n+1)/(2p')$ 
and $0<\sigma < (n+1)/2$. 
Let $f$ be a measurable function on $\h$ satisfying 
\[
\sup_{r >0,x\in \h} 
	\frac{r^\sigma}{|B(x,r)|}\int_{\h\cap B(x,r)} 
	\left( |f(y)| y_n^\beta \right)^p dy \le 1 . 
	\]
Then there exists a constant $C>0$ such that
	  \begin{eqnarray*}
  \sup_{\{r >0: B(x,r)\subset \h\}} 
	\frac{r^\sigma}{|B(x,r)|}\int_{B(x,r)} 
	\left( z_n^\beta J_2(z) \right)^{p^*} dz \le C .
 \end{eqnarray*} 
\end{lem}

\begin{proof}
By \holder's inequality and Lemma \ref{lem:1.2}, we have  
\begin{eqnarray} 
 \int_{B(x,x_n)} |f(y)| dy 
&\le& \left( \int_{B(x,x_n)} \left( y_n^\beta |f(y)|\right)^{p} dy 
\right)^{1/p}
\left( \int_{B(x,x_n)}  y_n^{-\beta p'} dy \right)^{1/p'} \nonumber \\
&\le& C x_n^{-\beta+ n/p'}  \left( \int_{B(x,x_n)} 		\left( y_n^\beta |f(y)| \right)^{p} dy \right)^{1/p} \nonumber \\
&\le& C x_n^{-\beta+ n-\sigma/p}  \label{(3.2)}
 \end{eqnarray} 
since $-\beta p'+n>(n-1)/2$. By (\ref{eq:3.3}), we obtain
\begin{eqnarray*} 
 x_n^\beta J_2(x) 
&\le& C x_n^{\alpha+\beta-n}  \int_{B(x,x_n)}  |f(y)| dy \le C x_n^{\alpha-\sigma/p} =Cx_n^{-\sigma/p^*}.
 \end{eqnarray*} 
If $0 < r <  x_n/2$, then 
 \begin{eqnarray*}
	\frac{r^\sigma}{|B(x,r)|}\int_{B(x,r)} 
	\left( z_n^\beta J_2(z) \right)^{p^*} dz
&\le& 
	C\frac{r^\sigma}{|B(x,r)|}\int_{B(x,r)} 
	z_n^{-\sigma} dz \\
	&\le& C \left( \frac{r}{x_n}  \right)^{\sigma}  \le C
 \end{eqnarray*}
 and if $x_n/2  \le r < x_n$, then Lemma \ref{lem:1.2} gives 
 \begin{eqnarray*}
 	\frac{r^\sigma}{|B(x,r)|}\int_{B(x,r)} 
	\left( z_n^\beta J_2(z) \right)^{p^*} dz
&\le&
C\frac{r^\sigma}{|B(x,r)|}    \int_{B(x,r)} z_n^{-\sigma}  dz \\
&\le& C\frac{x_n^\sigma}{|B(x,x_n)|}    \int_{B(x,x_n)} z_n^{-\sigma}  dz  \le C  
 \end{eqnarray*} 
when $\sigma < (n+1)/2$.  
\end{proof}

\begin{thm} \label{thm:3.1}
\sl
Let  $1/p^* = 1/p - \alpha/\sigma > 0$.
Suppose $\beta < (n+1)/(2p')$ and $0<\sigma < (n+1)/2$. 
Then there exists a constant $C>0$ such that
\begin{eqnarray*}
\sup_{\{r > 0: B(x,r) \subset \h\}}  
\frac{r^\sigma}{|B(x,r)|}\int_{B(x,r)}  
 \left(z_n^\beta I_{\h,\alpha} f(z)  \right)^{p^*} 
  dz  &\le& C 
\end{eqnarray*}
when $f\ge 0$ such that $\displaystyle \sup_{r >0,x\in \h} 
	\frac{r^\sigma}{|B(x,r)|}\int_{\h\cap B(x,r)} 
	\left( |f(y)| y_n^\beta \right)^p dy \le 1$. 
\end{thm}

\begin{proof}
Let $f$ be a measurable function on $\h$ such that
$$ \sup_{r >0,x\in \h} 
	\frac{r^\sigma}{|B(x,r)|}\int_{\h\cap B(x,r)} 
	\left( |f(y)| y_n^\beta \right)^p dy \le 1.
$$ 
By Lemmas \ref{lem:3.1} and \ref{lem:3.4}, we obtain for $r > 0$ such that $B(x,r) \subset \h$ 
\begin{eqnarray*}
&&  \frac{r^\sigma}{|B(x,r)|}\int_{B(x,r)}   
\left( z_n^\beta I_{\h,\alpha} f(z)\right)^{p^*} dz \\
&\le& C \frac{r^\sigma}{|B(x,r)|}\int_{B(x,r)} \left( z_n^\beta J_1(z) \right)^{p^*} dz+C \frac{r^\sigma}{|B(x,r)|}    \int_{B(x,r)}
\left( z_n^\beta J_2(z) \right) ^{p^*} dz \\
&\le& C,
\end{eqnarray*}
as required.
\end{proof}

\begin{rem}\label{rem:3.1}
\rm
If $f(y) = |y_n|^{-1} (1+|y|)^{-m} \chi_{\h}(y)$, then 
\begin{itemize}
\item[(1)]
$\displaystyle
 \sup_{r>0,x\in \h} {\frac{r^\sigma}{|B(x,r)|} \int_{\h\cap B(x,r)} \left( f(y) y_n^\beta \right)^p dy < \infty}$
when $(\beta - 1) p + 1 > 0$ and $0\le \sigma + (\beta-1) p  < \sigma+ mp - n$; 
\item[(2)]  $\displaystyle I_{\h,\alpha} 	f(x) \ge C\int_{B(x,x_n/2)}  |x-y|^{\alpha-n} f(y) dy 
\ge C x_n^{\alpha-1}  (1+|x|)^{-m} $ for  $x\in \h$;
\item[(3)]
$\displaystyle
  \frac{x_n^{\sigma_1}}{|B(x,x_n)|}\int_{B(x,x_n)}   
 \left( z_n^{\beta} I_{\h,\alpha} f(z) \right)^{p^*} dz  
 \ge C \frac{x_n^{\sigma_1}}{|B(x,x_n)|}\int_{B(x,x_n)}   
 z_n^{(\beta+\alpha-1)p^*} dz \\
\to \infty 
$
as $x_n \to 0$ and $x\in \h\cap B(0,1)$ 
when $\sigma_1 + (\beta + \alpha - 1) p^*  < 0$, and thus the sharpness of exponent $\sigma$ is seen in  Theorem \ref{thm:3.1} 
when $-\sigma/p \le  \ \beta -1 < - \sigma_1/p -  (\sigma-\sigma_1)\alpha/\sigma$; 
\item[(4)]
$\displaystyle \int_{\h\cap B(0,1)}  
\left( x_n^\beta I_{\h,\alpha} f(x)  \right)^{p^*} dx =  \infty$ 
when $(\beta + \alpha - 1) p^* + 1 \le 0$ or 
$ \left(-1/p  < \right) \beta - 1 \le-1/p + (\sigma^{-1}-1)\alpha $. 
\end{itemize}
Hence, in Theorem \ref{thm:3.1}, the assumption that 
$B(x,r) \subset \h$ is needed when   
 $1/p' < \beta \le 1/p' + (\sigma^{-1}-1)\alpha$ 
 and $0 < (\sigma^{-1} - 1) \alpha < (n-1)/(2p')$. 
\end{rem}

\section{Sobolev's inequality for double phase functionals} \label{s4}

In this section, we are concerned with Sobolev's inequality for $I_{\h,\alpha} f$ of functions in weighted Morrey spaces of the double phase functional  $\Phi(x,t)$.

\begin{thm} \label{thm:3.2}
\sl
Let  $1/q = 1/p - \theta/\sigma$, 
$1/p^* = 1/p - \alpha/\sigma > 0$ and 
$1/q^* = 1/q - \alpha/\sigma > 0$.
Suppose $\beta < (n+1)/(2p')$ and $0<\sigma < (n+1)/2$. 
Then there exists a constant $C>0$ such that
\begin{eqnarray*}
\sup_{\{r > 0: B(x,r) \subset \h\}}  
\frac{r^\sigma}{|B(x,r)|}\int_{B(x,r)}  
\Phi^* \left( z,z_n^\beta I_{\h,\alpha} f(z)  \right) 
  dz  &\le& C 
\end{eqnarray*}
when $f\ge 0$ such that $\displaystyle \sup_{r >0,x\in \h} 
	\frac{r^\sigma}{|B(x,r)|}\int_{\h\cap B(x,r)} 
	\Phi\left(y, f(y) y_n^\beta \right) dy \le 1$. 
\end{thm}

\begin{proof}
Let $f$ be a nonnegative measurable function on $\h$ such that
$$
\sup_{r >0,x\in \h} 
	\frac{r^\sigma}{|B(x,r)|}\int_{\h\cap B(x,r)} 
	\Phi(y, f(y) y_n^\beta) dy \le 1.
$$ 
First we see from Theorem \ref{thm:3.1} that
$$
\sup_{\{r > 0: B(x,r) \subset \h\}}  \frac{r^\sigma}{|B(x,r)|}\int_{B(x,r)}  
 (z_n^\beta I_{\h,\alpha} f(z))^{p^*}  dz \le C.
$$
Next we show that
$$
\sup_{\{r > 0: B(x,r) \subset \h\}}  \frac{r^\sigma}{|B(x,r)|}\int_{B(x,r)}  
 (b(z) z_n^\beta I_{\h,\alpha} f(z))^{q^*}  dz \le C.
$$
Note that 
  \begin{eqnarray*}
&& b(x) I_{\h,\alpha} f(x)  \\
&=& 
b(x) \int_{B(x,x_n/2)} |x-y|^{\alpha-n} f(y) dy  
+
 b(x) \int_{B(x,x_n)\setminus B(x,x_n/2)} |x-y|^{\alpha-n} f(y) dy  \\
&=& T_1(x)+T_2(x).
\end{eqnarray*}
Set $g(y) = |f(y)|y_n^\beta  \chi_{\h}(y)$ and $h(y) =b(y) f(y) y_n^\beta  \chi_{\h}(y)$ for simplicity. We have
\begin{eqnarray*}
T_1(x) &=& 
\int_{B(x,x_n/2)}  |x-y|^{\alpha-n} \{b(x)-b(y)\} f(y) dy  
+
  \int_{B(x,x_n/2)}  |x-y|^{\alpha-n} b(y) f(y) dy  \\
&\le&C \int_{B(x,x_n/2)} |x-y|^{\alpha-n+\theta} f(y) dy 
+
 \int_{B(x,x_n/2)}  |x-y|^{\alpha-n} b(y) f(y) dy \\
&\le&   C x_n^{-\beta}  
 \int_{B(x,x_n/2)}  |x-y|^{\alpha-n+\theta}  f(y) y_n^{\beta}  dy \\
&& + C x_n^{-\beta}
  \int_{B(x,x_n/2)}  |x-y|^{\alpha-n} b(y) f(y) y_n^{\beta}  dy \\
&\le& C x_n^{-\beta}I_{\alpha+\theta} g(x) +C x_n^{-\beta} I_{\alpha} h(x)
\end{eqnarray*}
and
\begin{eqnarray*}
T_2(x) 
&\le& C \int_{B(x,x_n)\setminus B(x,x_n/2)} |x-y|^{\alpha-n+\theta}
f(y) dy \\
&& +
\int_{B(x,x_n)\setminus B(x,x_n/2)} |x-y|^{\alpha-n} b(y) f(y) dy  \\
&=& CT_{21}(x) + T_{22}(x).
\end{eqnarray*}
Hence 
\begin{equation}\label{(2)}
b(z)  z_n^{\beta}I_{\h,\alpha}f(z) \le C I_{\alpha+\theta} g(z) +C I_{\alpha} h(z)+  C T_{21}(z) z_n^{\beta}+  T_{22}(z) z_n^{\beta}
\end{equation}
for $z \in B(x,r)$. By Lemma \ref{lem:3.4}, we obtain 
$$
\sup_{0<r<x_n} \frac{r^\sigma}{|B(x,r)|}  \int_{B(x,r)} \left( T_{21}(z) z_n^\beta \right)^{q^*} dz \le C
$$
and 
$$
\sup_{0<r<x_n} \frac{r^\sigma}{|B(x,r)|}  \int_{B(x,r)} \left( T_{22}(z) z_n^\beta \right)^{q^*} dz  \le C.
$$

By (\ref{(2)}) and Lemma \ref{lem:1.1}, we obtain for $r > 0$ such that $B(x,r) \subset \h$ 
\begin{eqnarray*}
&&    \frac{r^\sigma}{|B(x,r)|}\int_{B(x,r)}  
 (b(z) z_n^\beta I_{\h,\alpha} f(z))^{q^*}  dz \\
&\le& C \frac{r^\sigma}{|B(x,r)|}\int_{B(x,r)}(I_{\alpha+\theta} g(z))^{q^*}  dz
+C \frac{r^\sigma}{|B(x,r)|}\int_{B(x,r)}(I_{\alpha} h(z))^{q^*}  dz \\
&& +C \frac{r^\sigma}{|B(x,r)|}    \int_{B(x,r)} \left( T_{21}(z) z_n^\beta \right)^{q^*} dz + \frac{r^\sigma}{|B(x,r)|}    \int_{B(x,r)} \left( T_{22}(z) z_n^\beta \right)^{q^*} dz \\
&\le& C.
\end{eqnarray*}
Thus the proof is completed.
\end{proof}

\section{Weighted Campanato spaces for the double phase functionals}\label{s5}

In this section, we are concerned with Sobolev type inequalities for $I_{\h,\alpha} f$ in the Campanato setting.

For a measurable function $f$ on $\h$, $x\in \h$ and $0<r<x_n$, we set 
\[
	f_{B(x,r)}=\frac{1}{|B(x,r)|}\int_{B(x,r)}f(y)\, dy.
\] 
Set $g = f \chi_{B(x,x_n)}$ and 
 \[
	I_\alpha g(z) = \int_{\h} |z-y|^{\alpha-n}  g(y)\, dy.
\]

\begin{thm} \label{thm:3.3}
\sl
Suppose $\beta < (n+1)/(2p')$ and $\sigma = \alpha p < (n+1)/2$. 
If $0 < \ve < \min \{1,\alpha\}$ 
and $1/p_1 = 1/p - (\alpha-\ve)/\sigma = \ve/\sigma$ and $\beta p_1 > - (n+1)/2$, then there exists a constant $C>0$ such that
\begin{eqnarray*}
\sup_{\{r > 0: B(x,r) \subset \h\}}  
\frac{1}{|B(x,r)|}\int_{B(x,r)}  \left( z_n^\beta \left| I_{\h,\alpha} f(z)  -  (I_{\alpha} g)_{B(x,r)} \right| \right)^{p_1}    dz  &\le& C 
\end{eqnarray*}
when $g = f \chi_{B(x,x_n)}$ and 
\[
\sup_{r >0,x\in \h} 
	\frac{r^\sigma}{|B(x,r)|}\int_{\h\cap B(x,r)} \left( |f(y)| y_n^\beta \right)^p dy \le 1. 
	\]  
\end{thm}

\begin{proof}
Let $f$ be a nonnegative measurable function on $\h$ such that
\[
\sup_{r >0,x\in \h} 
	\frac{r^\sigma}{|B(x,r)|}\int_{\h\cap B(x,r)} \left( |f(y)| y_n^\beta \right)^p dy \le 1. 
	\]  
Let $0 < \ve < \min \{1,\alpha\}$. For $0 < r < x_n/4$
 and  $z \in B(x,r)$, we see that 
 \begin{eqnarray*}
&&    I_{\h,\alpha} f(z)  -  (I_{\alpha} g)_{B(x,r)} 
 \nonumber  \\
&=&   \frac{1}{|B(x,r)|} \int_{B(z,z_n)} \left(\int_{B(x,r)}
\{|z-y|^{\alpha-n} - |w-y|^{\alpha-n}\} dw \right)  f(y) dy    \nonumber  \\
&& {} +   \frac{1}{|B(x,r)|} \int_{B(x,r)} \left(\int_{B(z,z_n)} 
|w-y|^{\alpha-n} f(y) dy -\int_{B(x,x_n)} |w-y|^{\alpha-n} f(y) dy  \right) dw
   \nonumber  \\
   & = & I_1 + I_2 .
 \end{eqnarray*}
Note that 
  \begin{eqnarray*}    
 |I_1|   &\le& C\int_{B(x,2r)} |z-y|^{\alpha-n} |f(y)| dy 
  +   Cr \int_{B(z,z_n)\setminus B(x,2r)} 
  |z-y|^{\alpha-n-1} |f(y)| dy  \nonumber \\
   &\le& C\int_{B(z,3r)\cap B(z,z_n)} |z-y|^{\alpha-n} |f(y)| dy 
  +   Cr \int_{B(z,z_n)\setminus B(z,r)} 
  |z-y|^{\alpha-n-1} |f(y)| dy  \nonumber \\
   &\le& Cr^\ve \int_{B(z,3r)\cap B(z,z_n)} |z-y|^{\alpha-\ve-n} |f(y)| dy 
  +   Cr^\ve \int_{B(z,z_n)\setminus B(z,r)}   |z-y|^{\alpha-\ve-n} |f(y)| dy  \nonumber \\
  &\le& C r^\ve I_{\h,\alpha-\ve} |f|(z)   
\end{eqnarray*}
since $B(x, 2r) \subset B(z,z_n)$.  Moreover, 
 \begin{eqnarray*}    
&& \int_{B(z,z_n)} |w-y|^{\alpha-n} f(y) dy -\int_{B(x,x_n)} |w-y|^{\alpha-n} f(y) dy  \nonumber \\
&=& 
\int_{B(z,z_n)\setminus B(x,x_n)} 
|w-y|^{\alpha-n} f(y) dy -\int_{B(x,x_n)\setminus B(z,z_n)} |w-y|^{\alpha-n} f(y) dy   ,
 \end{eqnarray*}
 so that by \eqref{(3.2)}
\begin{eqnarray*}    
| I_2|  &\le&
Cx_n^{\alpha-n} \int_{B(z,z_n)\setminus B(x,x_n)} |f(y)| dy 
+ Cz_n^{\alpha-n} \int_{B(x,x_n)\setminus B(z,z_n)}|f(y)| dy \\
&\le&
Cz_n^{\alpha-n} \int_{B(z,z_n)} |f(y)| dy 
+ Cx_n^{\alpha-n} \int_{B(x,x_n)}|f(y)| dy \\
&\le& C z_n^{-\beta}+C x_n^{-\beta}   \\
&\le&  C z_n^{-\beta} 
 \end{eqnarray*}
since $\sigma = \alpha p$. Hence we find 
    \begin{eqnarray}    
 \left|  I_{\h,\alpha} f(z)  -  (I_{\alpha} g)_{B(x,r)}   \right|  
&\le& C r^\ve I_{\h,\alpha} |f|(z)  + Cz_n^{-\beta} \label{(3)}
 \end{eqnarray}
 for $0 < r < x_n/4$  and  $z \in B(x,r)$. 
 
For $x_n/4 < r < x_n$ and $z \in B(x,r)$, we see from \eqref{(3.2)} that 
\begin{eqnarray}
  &&\left| I_{\h,\alpha} f(z) 
 - (I_{\alpha} g)_{B(x,r)} \right|  \nonumber \\
 &\le& C\int_{B(z,z_n)} |z-y|^{\alpha-n} |f(y)| dy
 + Cx_n^{\alpha-n} \int_{B(x,x_n)} |f(y)| dy
  \nonumber \\
 &\le& Cr^\ve I_{\h,\alpha-\ve} |f|(z)  + C x_n^{-\beta}. \label{(4)} 
\end{eqnarray}
By \eqref{(3)}, \eqref{(4)}, Lemma \ref{lem:1.2} and Theorem \ref{thm:3.1}, we obtain 
 for $0 < r < x_n$ 
\begin{eqnarray*}
&&  \frac{1}{|B(x,r)|}\int_{B(x,r)}  \left( z_n^\beta \left| I_{\h,\alpha} f(z)  - (I_{\alpha} g)_{B(x,r)} \right| \right)^{p_1}    dz \\
&\le& C\frac{r^\sigma}{|B(x,r)|}\int_{B(x,r)} 
	\left( z_n^{\beta} I_{\h,\alpha-\ve} |f|(z) \right)^{p_1} dz 
	+  C \\
&\le& C  
\end{eqnarray*}
since $\ve p_1 = \sigma$, when $\beta p_1 > - (n+1)/2$. 

Thus this theorem is proved.
\end{proof}

Our second aim in this section is to establish the following result in the double phase setting.

\begin{thm} \label{thm:3.4}
\sl
Let  $1/q = 1/p - \theta/\sigma$
 and $\alpha q=\sigma = (\alpha+\theta)p$. 
 For $0 < \ve < \min \{1,\alpha\}$, set 
$1/q_1 = 1/q - (\alpha-\ve)/\sigma=\ve/\sigma >0$.
Suppose $\beta < (n+1)/(2p')$, $0<\sigma < (n+1)/2$ and  $\beta q_1 > - (n+1)/2$. 
Then there exists a constant $C>0$ such that
\begin{eqnarray*}
\sup_{\{r > 0: B(x,r) \subset \h\}}  
\frac{1}{|B(x,r)|}\int_{B(x,r)}  
\left(  z_n^\beta b(z) \left| I_{\h,\alpha} f(z)- (I_{\alpha} g)_{B(x,r)} \right|  \right)^{q_1}   dz  \le C 
\end{eqnarray*}
when $g = f \chi_{B(x,x_n)}$ and 
$$
 \sup_{r >0,x\in \h} 
	\frac{r^\sigma}{|B(x,r)|}\int_{\h\cap B(x,r)} 
	\Phi\left(y, |f(y)| y_n^\beta \right) dy \le 1.
$$ 
\end{thm}

\begin{proof}
Let $f$ be a measurable function on $\h$ such that
$$
 \sup_{r >0,x\in \h} 
	\frac{r^\sigma}{|B(x,r)|}\int_{\h\cap B(x,r)} 
	\Phi\left(y, |f(y)| y_n^\beta \right) dy \le 1.
$$ 
Let $0 < \ve < \min \{1,\alpha\}$. Then, for $0 < r < x_n/4$ and $z \in B(x,r)$, we see from the proof of Theorem \ref{thm:3.3} that  
\begin{eqnarray*}
   \left| I_{\h,\alpha} f(z)  - (I_{\alpha} g)_{B(x,r)} \right| 
   &\le& C r^\ve \int_{B(z,3r)\cap B(z,z_n)} |z-y|^{\alpha-\ve-n} |f(y)| dy   \\
&&{}  +   Cr^{\ve} \int_{B(z,z_n)\setminus B(z,r)}   |z-y|^{\alpha-\ve-n} |f(y)| dy   
 + |I_2|  \\
  &=& C(U_1+U_2) + |I_2|. 
\end{eqnarray*}
We have by \eqref{(3.2)}
\begin{eqnarray*}    
b(z)|I_2| &\le&
C  x_n^{\alpha-n+\theta} \int_{B(z,z_n)\setminus B(x,x_n)} |f(y)| dy 
+ C  x_n^{\alpha-n} \int_{B(z,z_n)\setminus B(x,x_n)} |b(y)f(y)| dy \\
&& {}+ C  z_n^{\alpha-n+\theta} \int_{B(x,x_n)\setminus B(z,z_n)}|f(y)| dy 
+ C  z_n^{\alpha-n} \int_{B(x,x_n)\setminus B(z,z_n)}|b(y)f(y)| dy \\
&\le&
C  x_n^{\alpha-n+\theta} z_n^{-\beta+n-\sigma/p}
+ C  z_n^{\alpha-n} z_n^{-\beta+n-\sigma/q} \\
&&{} +
C  z_n^{\alpha-n+\theta} x_n^{-\beta+n-\sigma/p}  
+ C  z_n^{\alpha-n} x_n^{-\beta+n-\sigma/q}\\
&\le&  C z_n^{-\beta} 
 \end{eqnarray*}
since $\sigma = (\alpha+\theta)p = \alpha q$. Hence we find 
\begin{eqnarray}
 b(z) \left| I_{\h,\alpha} f(z)  - (I_{\alpha} g)_{B(x,r)} \right| 
  &\le& Cb(z)(U_1+U_2) +C z_n^{-\beta} . 
 \label{(5)} 
\end{eqnarray}
Note that 
\begin{eqnarray*}
 b(z) U_1  
&=& 
r^\ve \int_{B(z,3r)\cap B(z,z_n)}  |z-y|^{\alpha-\ve-n} \{b(z)-b(y)\} |f(y)| dy  \\
&& +r^\ve  \int_{B(z,3r)\cap B(z,z_n)}  |z-y|^{\alpha-\ve-n} b(y) |f(y)| dy  \\
&\le&C r^{\ve} \int_{B(z,3r)\cap B(z,z_n)} |z-y|^{\alpha-\ve-n+\theta} |f(y)|  dy \\
&&+ r^{\ve}  \int_{B(z,3r)\cap B(z,z_n)}  |z-y|^{\alpha-\ve-n} b(y) |f(y)| dy \\
&\le& Cr^{\ve}  I_{\h,\alpha-\ve+\theta}|f|(z) +r^{\ve}  I_{\h,\alpha-\ve}(b|f|)(z).
\end{eqnarray*}
On the other hand,
\begin{eqnarray*}
b(z) U_2  
&=& 
r^\ve \int_{B(z,z_n)\setminus B(z,r)}  |z-y|^{\alpha-\ve-n} \{b(z)-b(y)\} |f(y)| dy  \\
&&+ r^\ve  \int_{B(z,z_n)\setminus B(z,r)}  |z-y|^{\alpha-\ve-n} b(y) |f(y)| dy  \\
&\le&C r^{\ve}\int_{B(z,z_n)\setminus B(z,r)}  |z-y|^{\alpha-\ve-n+\theta}  |f(y)| dy  \\
&&+  r^{\ve}  \int_{B(z,z_n)\setminus B(z,r)}  |z-y|^{\alpha-\ve-n} b(y) |f(y)| dy \\
&\le& Cr^{\ve}  I_{\h,\alpha-\ve+\theta}|f|(z) +r^{\ve}  I_{\h,\alpha-\ve}(b|f|)(z).
\end{eqnarray*}
Hence we find by (\ref{(5)})
    \begin{eqnarray}    
&& b(z) \left|  I_{\h,\alpha} f(z)  -  (I_{\alpha} g)_{B(x,r)}   \right|  \nonumber \\
&\le& C\left\{
r^{\ve}  I_{\h,\alpha-\ve+\theta}|f|(z) +r^{\ve}  I_{\h,\alpha-\ve}(b|f|)(z)  +  z_n^{-\beta} \right\} . \label{(5.4)}
 \end{eqnarray}
  
For $x_n/4 < r < x_n$, we see from the proof of Theorem \ref{thm:3.3}  that for $z \in B(x,r)$
\begin{eqnarray}
  && b(z) \left| I_{\h,\alpha} f(z)  - (I_{\alpha} g)_{B(x,r)} \right|  \nonumber \\
 &\le& C\int_{B(z,z_n)} |z-y|^{\alpha-n+\theta} |f(y)| dy+\int_{B(z,z_n)}  |z-y|^{\alpha-n} b(y) |f(y)| dy
 +  C x_n^{-\beta}  \nonumber \\
 &\le& C\left\{
 r^\ve I_{\h,\alpha-\ve+\theta} |f|(z)  +r^{\ve}  I_{\h,\alpha-\ve}(b|f|)(z)+  x_n^{-\beta} \right\}. \label{(5.5)} 
\end{eqnarray}
By  (\ref{(5.4)}),  (\ref{(5.5)}), Lemma \ref{lem:1.2} and Theorem \ref{thm:3.1} we have
\begin{eqnarray*}
 && \frac{1}{|B(x,r)|}\int_{B(x,r)} \left(  z_n^\beta b(z) \left| I_{\h,\alpha} f(z)  -  (I_{\h,\alpha} f)_{B(x,r)} \right|   \right)^{q_1} dz \\
&\le&C \biggl\{ \frac{r^{\sigma}}{|B(x,r)|}\int_{B(x,r)} |z_n^\beta I_{\h,\alpha-\ve+\theta}f(z) |^{q_1} dz \\
&& {} 
+\frac{r^{\sigma}}{|B(x,r)|}\int_{B(x,r)} |z_n^\beta  I_{\h,\alpha-\ve}(bf)(z) |^{q_1} dz
+ 1 \biggr\}
 \\
&\le&C
\end{eqnarray*} 
since $\ve q_1 = \sigma$, when $\beta q_1 > - (n+1)/2$. 

Thus the proof is completed.
\end{proof}


\section{Sobolev functions}

In this section, we are concerned with Sobolev type inequalities for Sobolev functions in the Campanato setting.

First let us show the following result.

\begin{lem}\label{lem:A1}
\sl
If $u\in C^1(\h)$ and $B(x,r) \subset \h$, then 
for $z\in B(x,r)$
\begin{eqnarray*}
  \left| u(z) -  u_{B(x,r)} \right| 
     \le  C\int_{B(x,r)} |z-w|^{1-n}  |\nabla u(w)| dw .
\end{eqnarray*} 
\end{lem}

\begin{proof}
By the mean value theorem for analysis we find 
\begin{eqnarray*}
  &&  \left| u(z) - u_{B(x,r)} \right| \\
     &=&  \left|  \frac{1}{|B(x,r)|}\int_{B(x,r)} \{u(z) - u(y)\} dy \right| \\
     &=&  \left|  \frac{1}{|B(x,r)|} \int_{B(x,r)}\left\{
     \int_0^1 \frac{d}{dt} u(z + t(y-z) dt \right\} dy \right| \\
      &\le&  \frac{1}{|B(x,r)|} \int_{B(x,r)}\left\{
     \int_0^1 |y-z| |\nabla u(z + t(y-z)| dt \right\} dy  \\
           &\le&  2r \frac{1}{|B(x,r)|} \int_0^1  \left\{
    \int_{B(x,r)} |\nabla u(z + t(y-z)| dy \right\} dt . 
 \end{eqnarray*}     
If $w =  z + t(y-z)$, then $|w-z| = t|y-z|  \le 2rt$, so that  
\begin{eqnarray*}
 \left| u(z) - u_{B(x,r)} \right| 
  &\le&  2r \frac{1}{|B(x,r)|}\int_{B(x,r)}  |\nabla u(w)| \left\{
     \int_{|w-z|/(2r)}^1 t^{-n} dt \right\} dw  \\
     &\le&  C\int_{B(x,r)} |z-w|^{1-n}  |\nabla u(w)| dw,
\end{eqnarray*} 
as required.
\end{proof}
 
\begin{thm} \label{thm:A1}
\sl
Suppose $1/p^* = 1/p - 1/\sigma > 0$, 
 $1-(n-1)/(2p')  -\sigma/p< \beta <  (n+1)/(2p')$ 
 and $(1  -(n-1)/(2p') -\sigma/p)p^* + (n+1)/2 > 0$.
 Then there exists a constant $C>0$ such that
\begin{eqnarray*}
\frac{r^\sigma}{|B(x,r)|}\int_{B(x,r)}  
 \left( z_n^\beta  \left| u(z) - u_{B(x,x_n)} \right| \right)^{p^*} 
  dz  &\le& C 
\end{eqnarray*}
for  $x\in \h$,  $0 < r < x_n$ and  
$u \in C^1(\h)$ with  
$$ 
\sup_{r > 0, x\in \h} 
	\frac{r^\sigma}{|B(x,r)|}\int_{\h\cap B(x,r)} 
	\left( |\nabla u(y)| y_n^\beta \right)^p dy \le 1.
$$ 
\end{thm}

\begin{rem}\label{rem:6.0}
\rm
The condition that $(1  -(n-1)/(2p') -\sigma/p)p^* + (n+1)/2 > 0$ 
is written as 
\[
((n+1)/2 -\sigma)/p^*   >  (n-1)/(2p') ,
\]
which holds at least near $p = 1$ when $1 < \sigma < (n+1)/2$. 
\end{rem}

For a proof of Theorem \ref{thm:A1} we note :    
if $z\in B(x,x_n)$, then  Lemma \ref{lem:A1} gives 
\begin{eqnarray}
 &&  \left| u(z) - u_{B(x,x_n)} \right| \nonumber \\
    & \le &  C  \int_{B(x,x_n)} |z-y|^{1-n}  |\nabla u(y)| dy  \nonumber \\
    & = &  C  \int_{B(z,z_n)} |z-y|^{1-n}  |\nabla u(y)| dy 
    +C  \int_{B(x,x_n)\setminus B(z,z_n)} |z-y|^{1-n}  |\nabla u(y)| dy  \nonumber    \\
     & = &  C\{ I_{\h,1}f(z) + I_2(z)\}   \label{(6)} ,
\end{eqnarray} 
where $f(y) = |\nabla u(y)| \chi_{\h}(y)$. 

\begin{lem}\label{lem:4.1}
\sl
Suppose  $\beta < (n+1)/(2p')$. 
Then there exists a constant $C>0$ such that
	  \begin{eqnarray*}
  	\int_{B(x,x_n)\cap B(z,r)}  	|f(y)| dy    
	&\le& C x_n^{(n-1)/(2p')}  r^{(n+1)/(2p')-\beta}  \\
	&& {} \times  \left( \int_{B(x,x_n)\cap B(z,r)} 	
	\left( y_n^\beta |f(y)|\right)^{p} dy 
\right)^{1/p}
 \end{eqnarray*} 
for $z \in B(x,x_n)$ and $r > z_n$. 
\end{lem}

\begin{proof}
For $z \in B(x,x_n)$ and $r > z_n$ we have by \holder's inequality 
  \begin{eqnarray*} 
&& \int_{B(x,x_n)\cap B(z,r)}  |f(y)| dy \\
&\le& \left( \int_{B(x,x_n)\cap B(z,r)} 		\left( y_n^\beta |f(y)|\right)^{p} dy 
\right)^{1/p}
\left( \int_{B(x,x_n)\cap B(z,r)}  y_n^{-\beta p'} dy \right)^{1/p'}.
 \end{eqnarray*} 
Here note 
  \begin{eqnarray*} 
 \int_{B(x,x_n)\cap B(z,r)}  y_n^{-\beta p'} dy   
&\le& C\int_0^{2r} 
 (\sqrt{x_n y_n})^{n-1} 	 y_n^{-\beta p'} dy_n \\
 &\le& Cx_n^{(n-1)/2}  r^{(n-1)/2-\beta p' + 1}  
 \end{eqnarray*} 
since $(n-1)/2-\beta p' + 1>0$. Therefore 
  \begin{eqnarray*} 
&& \int_{B(x,x_n)\cap B(z,r)}  |f(y)| dy \\
&\le&
C x_n^{(n-1)/(2p')}  r^{(n+1)/(2p')-\beta}   \left( \int_{B(x,x_n)\cap B(z,r)} 	
	\left( y_n^\beta |f(y)|\right)^{p} dy 
\right)^{1/p}, 
 \end{eqnarray*}
 which gives the result.  
\end{proof}

\begin{lem}\label{lem:4.2}
\sl
Suppose $\alpha-(n-1)/(2p') -\sigma/p< \beta <  (n+1)/(2p')$. 
Let $f$ be a measurable function on $\h$ satisfying 
\[
\sup_{x\in \h, 0<r<x_n} 
	\frac{r^\sigma}{|B(x,r)|}\int_{B(x,r)} 
	\left( |f(y)| y_n^\beta \right)^p dy \le 1 . 
	\]
Set 
\[
     J_{\alpha} |f|(z) = 
     \int_{B(x,x_n)\setminus  B(z,z_n)} 	|x-y|^{\alpha-n}    |f(y)| dy .
 \]
Then there exists a constant $C>0$ such that
	  \begin{eqnarray*}
  	z_n^\beta J_\alpha |f|(z) \le Cx_n^{(n-1)/(2p')}  
  	z_n^{\alpha -(n-1)/(2p') -\sigma/p}
 \end{eqnarray*} 
for $z\in B(x,x_n)$. 
\end{lem}

\begin{proof}
Let $f$ be a nonnegative measurable function on $\h$ satisfying 
\[
 \sup_{x\in \h, 0<r<x_n} 
	\frac{r^\sigma}{|B(x,r)|}\int_{B(x,r)} 
	\left( |f(y)| y_n^\beta \right)^p dy \le 1 . 
	\]
By Lemma \ref{lem:4.1}, we have  
\begin{eqnarray*} 
J_\alpha |f|(z) &\le& C \int_{B(x,x_n)\setminus  B(z,z_n)} 	
 |z-y|^{\alpha-n} |f(y)| dy 
\\
&\le& C \int_{z_n}^\infty  \left(	\frac{1}{|B(x,r)|}
 \int_{B(x,x_n)\cap  B(z,r)}  |f(y)| dy \right)  r^{\alpha-1} dr \\
&\le& Cx_n^{(n-1)/(2p')}   \int_{z_n}^\infty r^{-\beta +(n+1)/(2p')- n+(n-\sigma)/p}  r^{\alpha-1} dr
  \\
&\le&  Cx_n^{(n-1)/(2p')}  z_n^{\alpha- \beta  -(n-1)/(2p') -\sigma/p},
 \end{eqnarray*} 
since 
$\alpha-(n-1)/(2p') -\sigma/p< \beta <  (n+1)/(2p')$, 
which completes the proof. 
\end{proof}

\begin{lem}\label{lem:4.3}
\sl
Suppose $1/p^* = 1/p - \alpha/\sigma > 0$, 
 $\alpha-(n-1)/(2p')  -\sigma/p< \beta <  (n+1)/(2p')$ 
 and $(\alpha  -(n-1)/(2p') -\sigma/p)p^* > - (n+1)/2$. 
Let $f$ be a measurable function on $\h$ satisfying 
\[
 \sup_{x\in \h, 0<r<x_n} 
	\frac{r^\sigma}{|B(x,r)|}\int_{B(x,r)} 
	\left( |f(y)| y_n^\beta \right)^p dy \le 1 . 
	\]
Then there exists a constant $C>0$ such that
	  \begin{eqnarray*}
	\frac{r^\sigma}{|B(x,r)|}\int_{B(x,r)} 
	\left( z_n^\beta J_\alpha |f|(z) \right)^{p^*} dz \le C .
 \end{eqnarray*} 
 for $0 < r < x_n$. 
\end{lem}

\begin{proof}
By Lemma \ref{lem:4.2}, we have for $0 < r < x_n$ 
\begin{eqnarray*}
&& \frac{r^\sigma}{|B(x,r)|}\int_{B(x,r)} 
	\left( z_n^\beta J_\alpha |f|(z) \right)^{p^*} dz \\
&\le& Cx_n^{(n-1)p^*/(2p')}
  \frac{r^\sigma}{|B(x,r)|}\int_{B(x,r)} 
	z_n^{(\alpha  -(n-1)/(2p') -\sigma/p)p^*}  dz \\
&=& Cx_n^{(n-1)p^*/(2p')}
  \frac{r^\sigma}{|B(x,r)|}\int_{B(x,r)} 
	z_n^{-\sigma-(n-1)p^*/(2p')}  dz
\end{eqnarray*}
since $\alpha-(n-1)/(2p')  -\sigma/p< \beta <  (n+1)/(2p')$. If $0<r<x_n/2$, then 
$$
x_n^{(n-1)p^*/(2p')} \frac{r^\sigma}{|B(x,r)|}\int_{B(x,r)} z_n^{-\sigma-(n-1)p^*/(2p')}  dz
\le C \left(\frac{r}{x_n} \right)^\sigma \le C
$$
and if $x_n/2 \le r<x_n$, then Lemma \ref{lem:1.2} gives
\begin{eqnarray*}
&& x_n^{(n-1)p^*/(2p')}
  \frac{r^\sigma}{|B(x,r)|}\int_{B(x,r)} z_n^{-\sigma-(n-1)p^*/(2p')}  dz \\
&\le& C x_n^{(n-1)p^*/(2p')}
  \frac{x_n^\sigma}{|B(x,x_n)|}\int_{B(x,r)} z_n^{-\sigma-(n-1)p^*/(2p')}  dz \\
&\le& C
\end{eqnarray*}
since $(\alpha  -(n-1)/(2p') -\sigma/p)p^* > - (n+1)/2$. 

Thus we complete the proof.
\end{proof}

Now let us prove Theorem \ref{thm:A1}.

\begin{proof}[{\it Proof of Theorem \ref{thm:A1}}]
Let  $u \in C^1(\h)$ with  
$$ \sup_{r >0,x\in \h} 
	\frac{r^\sigma}{|B(x,r)|}\int_{\h\cap B(x,r)} 
	\left( |\nabla u(y)| y_n^\beta \right)^p dy \le 1.
$$ 
Set $f(y) = |\nabla u(y)| \chi_{\h}(y)$. 
By (\ref{(6)}), Theorem \ref{thm:3.1} and Lemma \ref{lem:4.3}, we obtain 
for $x\in \h$ and $0 < r < x_n$ 
\begin{eqnarray*}
&&  \frac{r^\sigma}{|B(x,r)|}\int_{B(x,r)}   \left( z_n^\beta \left| u(z) - u_{B(x,x_n)} \right|  \right)^{p^*} dz \\
&\le& C \frac{r^\sigma}{|B(x,r)|}\int_{B(x,r)} \left( z_n^\beta I_{\h,1}f(z) \right)^{p^*} dz+C \frac{r^\sigma}{|B(x,r)|}    \int_{B(x,r)}
\left( z_n^\beta I_2(z) \right)^{p^*} dz \\
&\le& C,
\end{eqnarray*}
as required.
\end{proof}

\begin{rem}\label{rem:4.1}
\rm
Let $u(x) = x_n^{-\ve}$. Then 
\begin{itemize}
\item[(1)]
$\displaystyle \sup_{x\in \h, 0<r<x_n} 
	\frac{r^\sigma}{|B(x,r)|}\int_{B(x,r)} 
	\left( |\nabla u(y)| y_n^\beta \right)^p dy < \infty$ 
when $\sigma = -(\beta -\ve-1)p < (n+1)/2$ ;  
\item[(2)]
$\displaystyle \frac{x_n^\sigma}{|B(x,x_n)|}\int_{B(x,x_n)} 
	\left( |u(y)| y_n^\beta \right)^{p^*} dy = \infty$ 
	when $(\beta -\ve)p^* \le - (n+1)/2$. 
\end{itemize}
Now $\ve$ is taken so that 
\[
     1 - (n+1)/(2p)  < \beta -\ve \le - (n+1)/(2p^*) 
     =  (n+1)/(2\sigma) - (n+1)/(2p)  
\]
when $p < \sigma < (n+1)/2$. In this case,  
\begin{itemize}
	\item[(3)]
	$\displaystyle \sup_{r > 0, x\in \h} 
	\frac{r^\sigma}{|B(x,r)|}\int_{\h\cap B(x,r)} 
	\left( |\nabla u(y)| y_n^\beta \right)^p dy = \infty$   
\end{itemize}
and we do not know whether Theorem \ref{thm:A1} 
holds under a weaker condition such as (1). 
\end{rem}

Our final goal is to obtain the following result in the double phase setting.

\begin{thm} \label{thm:A2}
\sl
Let  $1/q = 1/p - \theta/\sigma$, 
$1/p^* = 1/p - 1/\sigma > 0$ and 
$1/q^* = 1/q - 1/\sigma > 0$.
Suppose  
\[
- (n+1)/(2p^*) <1+\theta-(n-1)/(2p') -\sigma/p< \beta <  (n+1)/(2p') 
 \]
and 
 \[
 - (n+1)/(2q^*)  < 1-(n-1)/(2q')  -\sigma/q< \beta <  (n+1)/(2q') . 
\]
  Then there exists a constant $C>0$ such that
\begin{eqnarray*}
\frac{r^\sigma}{|B(x,r)|}\int_{B(x,r)}  
\Phi^* \left(z, z_n^\beta  \left| u(z) - u_{B(x,x_n)} \right| \right)   dz  &\le& C 
\end{eqnarray*}
for  $x\in \h$,  $0 < r < x_n$ and  
$u \in C^1(\h)$ with  \\
$$ \sup_{r >0,x\in \h} 
	\frac{r^\sigma}{|B(x,r)|}\int_{\h\cap B(x,r)} 
	\Phi \left( y, |\nabla u(y)| y_n^\beta \right) dy \le 1.
$$ 
\end{thm}

\begin{proof}
Let $u \in C^1(\h)$ with
$$ \sup_{r >0,x\in \h} 
	\frac{r^\sigma}{|B(x,r)|}\int_{\h\cap B(x,r)} 
	\Phi \left( y, |\nabla u(y)| y_n^\beta \right) dy \le 1.
$$
First we see from Theorem \ref{thm:A1} that
\begin{eqnarray*}
\frac{r^\sigma}{|B(x,r)|}\int_{B(x,r)}  
 \left( z_n^\beta  \left| u(z) - u_{B(x,x_n)} \right| \right)^{p^*} 
  dz  &\le& C. 
\end{eqnarray*}
Next we show that
\begin{eqnarray*}
\frac{r^\sigma}{|B(x,r)|}\int_{B(x,r)}  
 \left( b(z) z_n^\beta  \left| u(z) - u_{B(x,x_n)} \right| \right)^{q^*} 
  dz  &\le& C .
\end{eqnarray*}
Set $f(y) = |\nabla u(y)| \chi_{\h}(y)$. Recall from (\ref{(6)}) that
\begin{equation} \label{(7)}
  \left| u(z) - u_{B(x,x_n)} \right|  \le C\{ I_{\h,1}f(z) + I_2(z)\},
\end{equation} 
Note that 
\begin{eqnarray*}
&& b(z)  I_{\h,1}f(z)  \\
&=& 
\int_{B(z,z_n)}  |z-y|^{1-n} \{b(z)-b(y)\} |f(y)| dy  
+
  \int_{B(z,z_n)}  |z-y|^{1-n} b(y) |f(y)| dy  \\
&\le&C \int_{B(z,z_n)} |z-y|^{1-n+\theta} |f(y)|  dy 
+
 \int_{B(z,z_n)}  |z-y|^{1-n} b(y) |f(y)| dy \\
&\le&C I_{\h,1+\theta}|f|(z) +  I_{\h,1}(b|f|)(z).
\end{eqnarray*}
On the other hand,
\begin{eqnarray*}
&& b(z)  I_2(z)  \\
&=& 
\int_{B(x,x_n)\setminus B(z,z_n)}  |z-y|^{1-n} \{b(z)-b(y)\} |f(y)| dy  \\
&& +
  \int_{B(x,x_n)\setminus B(z,z_n)}  |z-y|^{1-n} b(y) |f(y)| dy  \\
&\le&C \int_{B(x,x_n)\setminus B(z,z_n)}  |z-y|^{1-n+\theta}  |f(y)| dy  
+
  \int_{B(x,x_n)\setminus B(z,z_n)}  |z-y|^{1-n} b(y) |f(y)| dy  \\
&=& C I_{21}(z) +I_{22}(z). 
\end{eqnarray*}

By Lemma \ref{lem:4.3}, we obtain 
$$
\sup_{0<r<x_n} \frac{r^\sigma}{|B(x,r)|}  \int_{B(x,r)} \left( I_{21}(z) z_n^\beta \right)^{q^*} dz \le C
$$
and 
$$
\sup_{0<r<x_n} \frac{r^\sigma}{|B(x,r)|}  \int_{B(x,r)} \left( I_{22}(z) z_n^\beta \right)^{q^*} dz  \le C.
$$

By (\ref{(7)}) and Theorem \ref{thm:3.1}, we obtain for $x\in \h$ and $0 < r < x_n$ 
\begin{eqnarray*}
&&    \frac{r^\sigma}{|B(x,r)|}\int_{B(x,r)}  
 (b(z) z_n^\beta \left| u(z) - u_{B(x,x_n)} \right|^{q^*}  dz \\
&\le& C \frac{r^\sigma}{|B(x,r)|}\int_{B(x,r)}   (b(z) z_n^\beta  I_{\h,1}f(z))^{q^*}  dz
+ C \frac{r^\sigma}{|B(x,r)|}\int_{B(x,r)}   (b(z) z_n^\beta  I_{2}(z))^{q^*}  dz \\
 &\le& C \frac{r^\sigma}{|B(x,r)|}\int_{B(x,r)}(z_n^\beta  I_{\h,1+\theta}f(z))^{q^*}  dz
+C \frac{r^\sigma}{|B(x,r)|}\int_{B(x,r)}(z_n^\beta  I_{\h,1}(b|f|)(z))^{q^*}  dz \\
&& +C \frac{r^\sigma}{|B(x,r)|}    \int_{B(x,r)} \left( I_{21}(z) z_n^\beta \right)^{q^*} dz +C \frac{r^\sigma}{|B(x,r)|}    \int_{B(x,r)} \left( I_{22}(z) z_n^\beta \right)^{q^*} dz \\
&\le& C.
\end{eqnarray*}
Thus the proof is completed.
\end{proof}

{\sc Acknowledgements.} We would like to express our thanks to the referees for their kind comments and helpful suggestions.


\vspace{3mm}

\it
\begin{flushright}
\begin{tabular}{c}
Department of Mathematics \\
Graduate School of Advanced Science and Engineering \\
Hiroshima University \\
Higashi-Hiroshima 739-8521, Japan \\
E-mail : yomizuta@hiroshima-u.ac.jp\\
and\\
Department of Mathematics  \\
Graduate School of Humanities and Social Sciences \\
Hiroshima University  \\
Higashi-Hiroshima 739-8524, Japan \\
E-mail : tshimo@hiroshima-u.ac.jp \\
\end{tabular}
\end{flushright}
\rm

\end{document}